\documentclass{svjour3}
\usepackage{chngcntr}
\usepackage{amsmath}
\usepackage{amssymb}

\usepackage{amsthm}

\usepackage{array}
\usepackage{hyperref}
\usepackage{enumitem}

\usepackage{pdflscape}

\usepackage{caption}

\usepackage{mathptmx}
\usepackage{helvet}
\usepackage{courier}

\usepackage{geometry}
\spnewtheorem{theorem}{Theorem}[section]{\bfseries}{\itshape}
\counterwithin{theorem}{section}
\spnewtheorem{corollary}[theorem]{Corollary}{\bfseries}{\itshape}
\spnewtheorem{lemma}[theorem]{Lemma}{\bfseries}{\itshape}
\spnewtheorem{proposition}[theorem]{Proposition}{\bfseries}{\itshape}
\spnewtheorem{definition}[theorem]{Definition}{\bfseries}{\itshape}% maybe \upshape?
\spnewtheorem{assumption}[theorem]{Assumption}{\bfseries}{\itshape}
\spnewtheorem{remark}[theorem]{Remark}{\bfseries}{\itshape}
\def\Ncantuba{{\mathbb N}}
\def\Fcantuba{{\mathbb F}}
\def\Zcantuba{{\mathbb Z}}

\def\lbrackCantuba{\left[}
\def\rbrackCantuba{\right]}
\def\adCantuba{{\rm ad}\ }

\def\algebHcantuba{\mathcal{H}}
\def\HeisenCantuba{\algebHcantuba(q)}

\def\HLieCantuba{\mathfrak{L}}
\def\HeisenCantubaLie{\HLieCantuba(q)}

\def\LieABcantuba{\lbrackCantuba A,B\rbrackCantuba}
\def\algIcantuba{I}
\def\qHeisenCantubaLieCantuba{{\overline{\HLieCantuba}(q)}}
\def\qHeisenCantuba{{\overline{\algebHcantuba}(q)}}

\def\baseAcantuba{\alpha}
\def\baseBcantuba{\beta}
\def\baseGcantuba{\gamma}

\def\obaseAcantuba{\overline{\baseAcantuba}}
\def\obaseBcantuba{\overline{\baseBcantuba}}
\def\obaseGcantuba{\overline{\baseGcantuba}}

\def\torsionHeisenCantuba{\algebHcantuba_p}
\def\torsionLieCantuba{\HLieCantuba_p}

\def\preHLieCantuba{\mathfrak{K}_p}

\def\derLieCantuba{\mathfrak{M}_p}

\def\gradsubCantuba{H}

\def\nonLiesubCantuba{\Gamma}

\def\Zcantubamod{\overline}

\begin{document}

\title{Torsion-type $q$-deformed Heisenberg algebra and its Lie polynomials}
%\title*{Torsion-type $q$-deformed Heisenberg algebra and its Lie polynomials}

\author{Rafael Reno Cantuba, Sergei Silvestrov}

\institute{Rafael Reno Cantuba
\at Mathematics and Statistics Department, De La Salle University, Manila, 2401 Taft Ave., Malate, Manila, 1004 Metro Manila, Philippines. \\ \email{rafael\_cantuba@dlsu.edu.ph}
\and
Sergei Silvestrov
\at Division of Applied Mathematics, School of Education, Culture and Communication (UKK),
M\"alardalen University, Box 883, 721 23 V\"aster\aa s, Sweden.
\email{sergei.silvestrov@mdh.se}}

%\author{Rafael Reno S. Cantuba\\
%De La Salle University, Manila, Philippines\\ \\
%Sergei D. Silvestrov\\
%Mal\"ardalen University, V\"{a}ster\aa s, Sweden}

\date{}

\maketitle

%\abstract*{ Given a scalar parameter $q$, the $q$-deformed Heisenberg algebra $\HeisenCantuba$ is the unital associative algebra with  two generators $A,B$ that satisfy the $q$-deformed commutation relation $AB-qBA=\algIcantuba$, where $\algIcantuba$ is the multiplicative identity. For $\HeisenCantuba$ of torsion-type, that is if $q$ is a root of unity, characterization is obtained for all the Lie polynomials in $A,B$ and basis and graded structure and commutation relations for associated Lie algebras are studied.}

\abstract{ Given a scalar parameter $q$, the $q$-deformed Heisenberg algebra $\HeisenCantuba$ is the unital associative algebra with  two generators $A,B$ that satisfy the $q$-deformed commutation relation $AB-qBA=\algIcantuba$, where $\algIcantuba$ is the multiplicative identity. For $\HeisenCantuba$ of torsion-type, that is if $q$ is a root of unity, characterization is obtained for all the Lie polynomials in $A,B$ and basis and graded structure and commutation relations for associated Lie algebras are studied.
}

%\keywords{$q$-deformed Heisenberg algebra, Lie polynomial, commutator, root of unity, torsion-type\\
%{\bf Mathematics Subject Classification (2010):} 17B60, 16S15, 81R50, 05A30}

%\subclass{17B60 \and 16S15 \and 17B37 \and 81R50}

\section{Introduction}

The main objects considered in this article are the $q$-deformed
Heisenberg algebras $\HeisenCantuba$, the parametric family of unital
associative algebras with two generators and defining commutation
relations
\begin{equation} \label{rel:qHeis}
  AB-qBA=I \text{.}
\end{equation}
When \(q=1\), this relation becomes \(AB-BA=I\), which is satisfied by the linear operators $A= \frac{d}{dx}: f(x) \mapsto f'(x)$ and $B=M_x:f(x)\mapsto xf(x)$ acting on suitable invariant linear spaces such as the linear spaces of all polynomials or all formal power series in an indeterminate $x$ with complex or real coefficients, or differentiable real-valued or complex-values functions in a real variable $x$ with usual definitions of derivative operator from calculus.
This follows from Leibnitz rule
\(\frac{d}{dx}\bigl( x \cdot f(x) \bigr) = x \frac{d}{dx} f(x) + \left( \frac{d}{dx} x \right) f(x)\).
Up to a constant scaling factor (involving Planck's constant) this is the Heisenberg
canonical commutation relation of Quantum Mechanics.
The $q$-deformed Heisenberg algebras are important in modern Quantum Physics and Mathematics
(see \cite{Hel} and references there). In noncommutative geometry, and investigations on quantum groups and quantum spaces,
the $q$-deformed Heisenberg algebras appears as one of the key examples and as a building block for other non-commutative objects.
In the calculus of $q$-difference operators and in $q$-difference equations---a subject whose history goes back well over one and a
half century, to Euler and Jackson---the $q$-deformed Heisenberg commutation relation plays the same fundamental role as the undeformed commutation relation in the
differential calculus and differential equations. Partly thanks to this, the most efficient way of obtaining many central results in
$q$-combinatorics and the theory of $q$-special functions is to make use of $q$-deformed Heisenberg algebras and their representations.

Whenever the scalar field has characteristic $0$, the undeformed Heisenberg algebra ($q=1$) is a simple algebra in the sense that the only two-sided ideals are the zero ideal and the whole algebra. On the other hand, the $q$-deformed Heisenberg algebra for $q\neq 1$ is not simple and has many two-sided nontrivial ideals. If moreover $q\neq 1$ is a root of unity, then the center consisting of all elements commuting with all elements of the algebra is nontrivial, that is consists not only of constant multiples of the identity element
\cite{Hel,Hel05}.

Any associative algebra yields a Lie algebra when associative multiplication is replaced with the commutator as a Lie bracket
and any subset of the associative algebra generates either the whole Lie algebra or a proper Lie subalgebra.
If the associative algebra is infinite-dimensional, then the resulting Lie algebra can be either finite-dimensional or infinite-dimensional and can have in general complicated structure and properties as a Lie algebra. It is an important and interesting problem to characterize all elements of an associative algebra belonging to a
Lie subalgebra generated by a given subset of the associative algebra. Any element in the Lie subalgebra can be uniquely represented as a finite linear combination of elements in any basis of the Lie subalgebra, and so, for a given subset of the associative algebra, the description of the bases in the Lie subalgebra generated by this subset yields a description of all elements of the associative algebra belonging to the Lie subalgebra, as finite linear combinations of the basis elements. We show in this work this interesting property of the Lie algebra being studied that the elements can be represented in other ways than using linear combinations of iterated commutators of the generators.  This involves the fact that in noncommutative associative algebras defined by generators and relations, the elements are given often in the form of noncommutative polynomial expressions in generators inherited from the free algebra with proper identification of equal elements due to reductions following from commutation relations. In order to determine whether an element of the associative algebra belongs to the Lie subalgebra one needs to describe a basis in this Lie subalgebra and then determine whether and how this given element can be expressed as a linear combination of the elements of that basis or not. Thus, formulas of rewriting elements when possible using repeated commutators and linear combinations, starting from the initial generating set of the Lie subalgebra, are important as well.

The Lie subalgebras for $q$-deformed Heisenberg algebra $\HeisenCantuba$ when $q$ is not a root of unity has been considered in \cite{Can17} where especially the Lie subalgebra generated by generators $A$ and $B$ has been studied in more detail and in particular its bases and some properties have been described.
This paper is devoted to Lie subalgebras for $q$-deformed Heisenberg algebras $\HeisenCantuba$ when $q$ is a root of unity, which is a natural continuation of \cite{Can17}.
Specifically, for the Lie subalgebra generated by the generators $A$ and $B$ the basis is described in terms of $A$ and $B$, and thus a characterization of elements of the Lie subalgebra generated by them is achieved.
The formulas allowing to rewrite elements written in a standard normal form of $\HeisenCantuba$ in terms of these basis elements of this Lie subalgebra are obtained and some properties of the basis and graded structure of the Lie subalgebra are described.

\section{Preliminaries}

Let $\Fcantuba$ be a field, and let $\mathcal{A}$ be a unital associative $\Fcantuba$-algebra, which we simply call an algebra throughout. We turn $\mathcal{A}$ into a Lie algebra over $\Fcantuba$ with Lie bracket given by $\lbrackCantuba f,g\rbrackCantuba := fg-gf$ for all $f,g\in\mathcal{A}$. We reserve the term \emph{subalgebra} to refer to a subset of $\mathcal{A}$ that is also an algebra using the same operations, and so that this kind of a substructure under the algebra structure is distinguished from the corresponding substructure under the Lie algebra structure, we use the term \emph{Lie subalgebra} to mean a subset of $\mathcal{A}$ that is also a Lie algebra over $\Fcantuba$ under the same Lie algebra operations. We treat the terms \emph{ideal} and \emph{Lie ideal} similarly. However, the term \emph{derived algebra} refers to a Lie algebra structure, as we are not interested in any analogue of it in the associative structure. 

Given $f_1,f_2,\ldots,f_n\in\mathcal{A}$, we define the Lie subalgebra of $\mathcal{A}$ generated by $f_1,f_2,\ldots,f_n$ as the smallest Lie subalgebra $\mathcal{B}$ of $\mathcal{A}$ that contains $f_1,f_2,\ldots,f_n$; i.e., if $\mathcal{C}$ is a Lie subalgebra of $\mathcal{A}$, with $f_1,f_2,\ldots,f_n\in\mathcal{C}$ and that $\mathcal{C}$ is contained in $\mathcal{B}$, then $\mathcal{B}=\mathcal{C}$. The elements of $\mathcal{B}$ are called the \emph{Lie polynomials in $f_1,f_2,\ldots,f_n$}. Denote by $\algIcantuba$ the unity element of $\mathcal{A}$. Given any nonzero element $f$ of $\mathcal{A}$, we interpret $f^0$ as $\algIcantuba$. Given $f\in\mathcal{A}$ the linear map $\adCantuba f:\mathcal{A}\rightarrow\mathcal{A}$ is defined by $g\mapsto\lbrackCantuba f,g\rbrackCantuba$.

%\subsection{The $q$-deformed Heisenberg algebra}\label{Hqsubsec}

Fix a $q\in\Fcantuba$. The $q$-deformed Heisenberg algebra is the unital associative $\Fcantuba$-algebra $\HeisenCantuba$ generated by two elements $A,B$ satisfying the relation $AB-qBA=\algIcantuba$. By a simple application of \cite[Lemma 1.7]{Reut}, the relation $AB-qBA=\algIcantuba$ cannot be expressed in terms of only Lie polynomials in $A,B$ for all cases for $q$ considered in \cite{Can17} and also in this work.  Denote the set of all nonnegative integers by $\Ncantuba$, and the set of all positive integers by $\Zcantuba^+$.  If $q\notin\{0,1\}$, then by \cite[Corollary 4.5]{Hel05}, the elements 

\begin{eqnarray}\label{Heisenbasis}
\LieABcantuba^k,\quad\LieABcantuba^k A^l,\quad B^{l}\LieABcantuba^k, \quad (k\in\Ncantuba, l\in\Zcantuba^+),
\end{eqnarray}
form a basis for $\HeisenCantuba$.

Define $\{0\}_q:=0$, and for each $n\in\Zcantuba^+$, we recursively define $\{n\}_q:=1+q\{n-1\}_q$. That is, $\{n\}_q=1+q+\cdots+q^{n-1}$. If $q\neq 1$, then $\{n\}_q=\frac{1-q^n}{1-q}$. The \emph{Gaussian binomial coefficients} or \emph{$q$-binomial coefficients} are recursively defined by
\begin{eqnarray}
{{n}\choose{0}}_q & = & 1,\\
{{0}\choose{k+1}}_q & = & 0,\\
{{n+1}\choose{k+1}}_q & = & {{n}\choose{k}}_q + q^{k+1}{{n}\choose{k+1}}_q,\label{qbirec}
\end{eqnarray}
for any $n,k\in\Ncantuba$. These $q$-binomial coefficients satisy the symmetry property ${{n}\choose{k}}_q={{n}\choose{n-k}}_q$ for any $k\in\{0,1,\ldots,n\}$, and as a consequence also the properties
\begin{eqnarray}
{{n}\choose{0}}_q  =   {{n}\choose{n}}_q  & = &   1,\label{qbi1}\\
{{n}\choose{1}}_q =   {{n}\choose{n-1}}_q  & = &  \{n\}_q.\label{qbi2}
\end{eqnarray}

\subsection{Structure constants of the $q$-deformed Heisenberg algebra}\label{structconstSec}

If $\mathcal{A}$ is an algebra with basis $\{\beta_j\   :\   j\in J\}$ for some index set $J$, then it is worthwhile to know how the product of any two basis elements $\beta_j$ and $\beta_k$ can be expressed as a linear combination of $\{\beta_j\   :\   j\in J\}$, i.e., $\beta_j\beta_k=\sum_{i}e_i(j,k)\beta_i$ for some scalars $e_i(j,k)$. We refer to the scalars $e_i(j,k)$ as the \emph{structure constants} of the algebra $\mathcal{A}$ with respect to the basis $\{\beta_j\   :\   j\in J\}$ of $\mathcal{A}$. In this subsection, we illustrate how to obtain the structure constants of $\HeisenCantuba$ with respect to its basis \eqref{Heisenbasis}. From \cite[Equations~(18),(19),(39)]{Hel05} and from \cite[Proposition~3.3]{Can17}, an algorithm for expressing the product of any two elements in \eqref{Heisenbasis} as a linear combination of \eqref{Heisenbasis} can be completely determined using the relations

\begin{eqnarray}
A^l\LieABcantuba^k & = & q^{kl}\LieABcantuba^kA^l,\label{moveAi}\\
\LieABcantuba^kB^l & = & q^{kl} B^l\LieABcantuba^k,\label{moveBj}\\
A^lB^l & = &  (q-1)^{-l}\sum_{i=0}^l(-1)^{l-i}q^{{i+1}\choose{2}}{{l}\choose{i}}_q\LieABcantuba^i.\label{equalAlBl}\\
B^lA^l & = & q^{-{{l}\choose{2}}}(q-1)^{-l}\sum_{i=0}^l(-1)^{l-i}q^{{l-i}\choose{2}}{{l}\choose{i}}_q\lbrackCantuba A,B\rbrackCantuba^i.\label{equalBlAl}
\end{eqnarray}
All such relations are consequences of the simple relation $AB-qBA=\algIcantuba$. The basis elements from \eqref{Heisenbasis} are essentially of the following three types:

\begin{eqnarray}
\LieABcantuba^k, \label{typeLie}\\
\LieABcantuba^k A^l,\label{typeA0}\\
B^{l}\LieABcantuba^k, \label{typeB0}
\end{eqnarray}
where $k\in\Ncantuba, l\in\Zcantuba^+$. We discuss in this subsection products of two basis elements from \eqref{Heisenbasis} under different cases based on whether each factor is of the type \eqref{typeLie}, \eqref{typeA0}, or \eqref{typeB0}. The simplest cases are when we have a product of two basis elements both of type \eqref{typeLie}, the case when we have a basis element of type \eqref{typeLie} multiplied by a basis element of type \eqref{typeA0}, and the case when we have a basis element of type \eqref{typeB0} multiplied by a basis element of type \eqref{typeLie}. That is, if we let $m,k\in\Ncantuba$ and $n,l\in\Zcantuba^+$ to be arbitrary, we have
\begin{eqnarray}
\LieABcantuba^m\cdot\LieABcantuba^k & = & \LieABcantuba^{m+k},\label{simple1}\\
\LieABcantuba^m\cdot\LieABcantuba^kA^l & = & \LieABcantuba^{m+k}A^l,\label{simple2}\\
B^n\LieABcantuba^m\cdot\LieABcantuba^k & = & B^n\LieABcantuba^{m+k}.\label{simple3}
\end{eqnarray}

As we consider more complicated cases, the relations \eqref{moveAi} to \eqref{equalBlAl} will turn out to be useful. The relation \eqref{moveAi} is relevant when we have a product of a basis element of type \eqref{typeA0} multiplied by a basis element of type \eqref{typeLie}, and also in the case when we have a product of two basis elements of type \eqref{typeA0}. More explicitly, we have
\begin{eqnarray}
\LieABcantuba^mA^n\cdot\LieABcantuba^k & = & q^{nk}\LieABcantuba^{m+k}A^n,\label{medA1}\\
\LieABcantuba^mA^n\cdot\LieABcantuba^kA^l & = & q^{nk}\LieABcantuba^{m+k}A^{n+l}.\label{medA2}
\end{eqnarray}
For the case when we have a basis element of type \eqref{typeLie} multiplied by a basis element of type \eqref{typeB0}, and the case when we have a product of both with type \eqref{typeB0}, the relation \eqref{moveBj} can be used to obtain
\begin{eqnarray}
\LieABcantuba^m\cdot B^l\LieABcantuba^k & = & q^{ml}B^l\LieABcantuba^{m+k},\label{medB1}\\
B^n\LieABcantuba^m\cdot B^l\LieABcantuba^k & = & q^{ml}B^{n+l}\LieABcantuba^{m+k}.\label{medB2}
\end{eqnarray}

The remaining cases involve less simple computations involving the relations \eqref{equalAlBl}, \eqref{equalBlAl}. For notational convenience, for each $l\in\Zcantuba^+$ and each $i\in\{0,1,\ldots,l\}$ we define
\begin{eqnarray}
 c_i(l) & := & (q-1)^{-l}(-1)^{l-i}q^{{i+1}\choose{2}}{{l}\choose{i}}_q,\label{cidef}\\
 d_i(l) & := & q^{-{{l}\choose{2}}}(q-1)^{-l}(-1)^{l-i}q^{{l-i}\choose{2}}{{l}\choose{i}}_q,\label{didef}
\end{eqnarray}
and so the relations \eqref{equalAlBl}, \eqref{equalBlAl} can be rewritten as
\begin{eqnarray}
A^lB^l & = & \sum_{i=0}^lc_i(l)\cdot\LieABcantuba^i,\label{newAlBl}\\
B^lA^l & = & \sum_{i=0}^ld_i(l)\cdot\LieABcantuba^i,\label{newBlAl}
\end{eqnarray}
respectively. Consider the product $\LieABcantuba^mA^n\cdot B^l\LieABcantuba^k$.
If $n\geq l$, we make use of the fact that $\LieABcantuba^mA^n\cdot B^l\LieABcantuba^k=\LieABcantuba^mA^{n-l}(A^l B^l)\LieABcantuba^k$ where we replace $A^l B^l$ using \eqref{newAlBl} to obtain
\begin{eqnarray}
\LieABcantuba^mA^n\cdot B^l\LieABcantuba^k = \sum_{i=0}^lc_i(l)\cdot\LieABcantuba^m(A^{n-l}\LieABcantuba^{i+k}),\nonumber
\end{eqnarray}
and we rewrite $A^{n-l}\LieABcantuba^{i+k}$ using \eqref{moveAi} to obtain
\begin{eqnarray}
\LieABcantuba^mA^n\cdot B^l\LieABcantuba^k & = & \sum_{i=0}^lq^{(i+k)(n-l)}c_i(l)\cdot\LieABcantuba^{m+i+k}A^{n-l},\quad\quad(n\geq l).\label{ABnbigger}
\end{eqnarray}
If $n<l$, we rewrite the product as  $\LieABcantuba^mA^n\cdot B^l\LieABcantuba^k=\LieABcantuba^m(A^n B^n)B^{l-n}\LieABcantuba^k$, and replace $A^nB^n$ using \eqref{newAlBl}. This results to
\begin{eqnarray}
\LieABcantuba^mA^n\cdot B^l\LieABcantuba^k = \sum_{i=0}^nc_i(n)\cdot(\LieABcantuba^{m+i}B^{l-n})\LieABcantuba^{k},\nonumber
\end{eqnarray}
in which we rewrite $\LieABcantuba^{m+i}B^{l-n}$ using \eqref{moveBj}, and we get
\begin{eqnarray}
\LieABcantuba^mA^n\cdot B^l\LieABcantuba^k & = & \sum_{i=0}^nq^{(m+i)(l-n)}c_i(n)\cdot B^{l-n}\LieABcantuba^{m+i+k},\quad\quad(n< l).\label{ABnsmaller}
\end{eqnarray}
Similar computations can be used for the product $B^n\LieABcantuba^m\cdot\LieABcantuba^k A^l$. The relations that show the structure constants are the following.
\begin{eqnarray}
B^n\LieABcantuba^m\cdot\LieABcantuba^k A^l& = & \sum_{i=0}^lq^{-(m+k)l}d_i(l)\cdot B^{n-l}\LieABcantuba^{m+k+i},\quad\quad(n\geq l)\label{BAnbigger},\\
B^n\LieABcantuba^m\cdot\LieABcantuba^k A^l & = & \sum_{i=0}^nq^{-(m+k)n}d_i(n)\cdot \LieABcantuba^{i+m+k}A^{l-n},\quad\quad(n< l).\label{BAnsmaller}
\end{eqnarray}
We summarize in Table~\ref{table:structconst} the relations we have discussed in this subsection that give the structure constants of $\HeisenCantuba$ with respect to the basis \eqref{Heisenbasis}.\\
\begin{center}
\scalebox{0.73}{
\begin{tabular}{|c|c|c|c|}
\hline
$\cdot$ &  $\LieABcantuba^k$ & $\LieABcantuba^kA^l$ & $B^l\LieABcantuba^k$  \\
\hline
$\LieABcantuba^m$  & \eqref{simple1} & \eqref{simple2} & \eqref{medB1} \\
\hline
$\LieABcantuba^mA^n$  & \eqref{medA1} & \eqref{medA2} & \eqref{ABnbigger}, \eqref{ABnsmaller} \\
\hline
$B^n\LieABcantuba^m$  & \eqref{simple3} & \eqref{BAnbigger}, \eqref{BAnsmaller} & \eqref{medB2} \\
\hline
\end{tabular}}
\captionof{table}{Relations that can be used to obtain the structure constants of $\HeisenCantuba$ with respect to the basis \eqref{Heisenbasis} \\}\label{table:structconst}
%\end{table}
\end{center}

The computational techniques described in this section are related to the $\Zcantuba$-gradation of $\HeisenCantuba$ from \cite[Corollary 4.5]{Hel05}. This $\Zcantuba$-gradation is the collection $\{\gradsubCantuba_n\  :\  n\in\Zcantuba\}$  of subspaces of $\HeisenCantuba$ such that for each $l\in\Zcantuba^+$, a basis for $\gradsubCantuba_l$ consists of all elements of the form $B^l\LieABcantuba^k$ ($k\in\Ncantuba$), and a basis for $\gradsubCantuba_{-l}$ consists of all elements of the form $\LieABcantuba^kA^l$ ($k\in\Ncantuba$). The subalgebra $\gradsubCantuba_0$ of $\HeisenCantuba$ has a basis consisting of all elements of the form $\LieABcantuba^k$ ($k\in\Ncantuba$).

\subsection{Lie polynomials when $q$ is not a root of unity}\label{notrootSec}
Denote by $\HeisenCantubaLie$ the Lie subalgebra of $\HeisenCantuba$ generated by $A,B$. Following the notation in \cite[Section 5]{Can17}, if $q$ is nonzero and not a root of unity, we denote $\HeisenCantuba$ by $\qHeisenCantuba$, and $\HeisenCantubaLie$ by $\qHeisenCantubaLieCantuba$. Then by \cite[Lemma 5.1]{Can} the elements in \eqref{Heisenbasis}, except any power of $A$ or $B$ with exponent not equal to $1$, can be expressed as elements of $\qHeisenCantubaLieCantuba$. Such elements are
\begin{eqnarray}\label{HeisenLiebasis}
A,\  B,\quad\LieABcantuba^k,\quad\LieABcantuba^k A^l,\quad B^{l}\LieABcantuba^k, \quad (k, l\in\Zcantuba^+).
\end{eqnarray}
Furthermore, the above elements form a basis for $\qHeisenCantubaLieCantuba$ \cite[Theorem 5.8]{Can17}. It was also shown in \cite{Can17} that $\qHeisenCantubaLieCantuba$ is a Lie ideal of $\qHeisenCantuba$. Thus, if we compute the Lie bracket of an element in \eqref{HeisenLiebasis} with an element in \eqref{Heisenbasis}, whether the latter is in \eqref{HeisenLiebasis} or not, then the result is a linear combination of \eqref{HeisenLiebasis}. A further result in \cite{Can17} is that the resulting quotient Lie algebra $\qHeisenCantuba/\qHeisenCantubaLieCantuba$ is one-step nilpotent. This implies that the Lie bracket of any two elements in \eqref{Heisenbasis} that are both not in \eqref{HeisenLiebasis} can be expressed uniquely as a linear combination of \eqref{HeisenLiebasis}. The Lie polynomials for the case $q=0$ are discussed in \cite[Section~4]{Can17}, while the case $q=1$ leads to a trivial low-dimensional Lie algebra as discussed in \cite[Section 1]{Can17}.

\section{Consequences of $q$ being a root of unity on structure constants and commutators}\label{ConSec}

Throughout, we assume that $q$ is a root of unity, and we denote by $p$ the least positive integer that satisfies the equation $q^p=1$.  As mentioned in Section~\ref{notrootSec}, we have a trivial Lie algebra of Lie polynomials when $q=1$, and so we further assume that $q\neq 1$, and hence $p\geq 2$. Following the terminology in \cite[Definition 5.2]{Hel}, the said conditions mean that $q$ is of \emph{torsion type with order} $p$. To remind us of these restrictions, we denote $\HeisenCantuba$ and $\HeisenCantubaLie$ by $\torsionHeisenCantuba$ and $\torsionLieCantuba$, respectively. We extend the terminology by calling $\torsionLieCantuba$ a \emph{torsion-type $q$-deformed Heisenberg algebra with order $p$}.

We make an important note here that the relation $AB-qBA=\algIcantuba$ still implies, by the Diamond Lemma \cite[Theorem 1.2]{Berg} and also by \cite[Theorem 3.1]{Hel} that the elements \eqref{Heisenbasis} still form a basis for $\torsionHeisenCantuba$.

Since $q\neq 1$, we have $\{n\}_q=\frac{1-q^n}{1-q}$. Let
\begin{eqnarray}
\Zcantubamod{n}=\min\{x\in\Ncantuba\  :\  x\equiv n\mod p\}.\label{remnot}
\end{eqnarray}
Then $n=Np+\Zcantubamod{n}$ for some integer $N$, and since $q^p=1$, we further have $\{n\}_q=\frac{1-q^{\Zcantubamod{n}}}{1-q}$, which is nonzero if $\Zcantubamod{n}\neq 0$ by the minimality of $p$. Thus,
\begin{eqnarray}
\{n\}_q = 0\quad\quad \Leftrightarrow\quad\quad n\equiv 0\mod p.\label{qmod}
\end{eqnarray}
Gvien $l\in\Zcantuba^+$ and $i\in\{1,2,\ldots,l\}$, the consequences on the $q$-binomial coefficient ${{l}\choose{i}}_q$ depends on comparison of $l$ with $p$. Suppose that $l<p$, by the properties of $q$-binomial coefficients, (see for instance \cite[p. 185]{Hel}) we have the identity
\begin{eqnarray}
{{l}\choose{i}}_q = \frac{(1-q^{l})(1-q^{l-1})\cdots(1-q^{l-i+1})}{(1-q)(1-q^{2})\cdots(1-q^{i})},\label{nonzerofactors}
\end{eqnarray}
where each factor $1-q^x$ in \eqref{nonzerofactors} satisfies $1\leq x<p$. By the minimality of $p$, we have ${{l}\choose{i}}_q \neq 0$ whenever $l<p$, $1\leq i<p$.

Suppose $l=p$. Then the numerator of \eqref{nonzerofactors} becomes zero, while the denominator is still nonzero because all the factors $1-q^y$ in there satisfy $1\leq y<p$. Then by further using \eqref{qbi1}, \eqref{qbi2}, \eqref{qmod}, we find that ${{p}\choose{i}}_q$ is $1$ if $i\in\{0,p\}$ and is $0$ if $i\in\{1,2,\ldots,p-1\}$. For the case $l=p+1$, we use \eqref{qbirec}, and by some routine computations, it can be shown that ${{p+1}\choose{i}}_q$ is also $1$ if $i\in\{0,p\}$ and $0$ if $i\in\{1,2,\ldots,p-1\}$. This can be routinely extended to any $l\geq p$ by induction. By these comments we have

\begin{eqnarray}
{{l}\choose{i}}_q & \neq & 0,\quad\quad\quad\   \  \  (l<p,\quad  i\in\{0,1,\ldots,l\}),\label{qbireduce1}\\
{{l}\choose{i}}_q & = & \left\{  \begin{array}{cc}
 &
    \begin{array}{cc}
      0, &  \quad\quad(l\geq p,\quad  i\in\{1,2,\ldots,l-1\}), \\
     1, &  (l\geq p,\quad  i\in\{0,l\}).
    \end{array}
\end{array}\right.\label{qbireduce2}
\end{eqnarray}

\subsection{Consequences on the structure constants}

We now discuss the consequences of the facts explained in the beginning of Section~\ref{ConSec} above about the structure constants of $\torsionHeisenCantuba$ with respect to the basis \eqref{Heisenbasis}. Recall that the structure constants can be obtained from the relations \eqref{simple1}--\eqref{medB2} and \eqref{ABnbigger}--\eqref{BAnsmaller}. Based on the appearance of \eqref{simple1}--\eqref{simple3}, the structure constants in these relations are not affected by $q$ being a root of unity. As for the relations \eqref{medA1}--\eqref{medB2}, the appearance of the structure constants suggests that we simply have to reduce the exponents to the smallest nonnegative representatives of these integers modulo $p$. We still use our notation \eqref{remnot} for this, and so \eqref{medA1}--\eqref{medB2} become

\begin{eqnarray}
\LieABcantuba^mA^n\cdot\LieABcantuba^k & = & q^{\Zcantubamod{nk}}\LieABcantuba^{m+k}A^n,\label{medA1p}\\
\LieABcantuba^mA^n\cdot\LieABcantuba^kA^l & = & q^{\Zcantubamod{nk}}\LieABcantuba^{m+k}A^{n+l},\label{medA2p}\\
\LieABcantuba^m\cdot B^l\LieABcantuba^k & = & q^{\Zcantubamod{ml}}B^l\LieABcantuba^{m+k},\label{medB1p}\\
B^n\LieABcantuba^m\cdot B^l\LieABcantuba^k & = & q^{\Zcantubamod{ml}}B^{n+l}\LieABcantuba^{m+k}.\label{medB2p}
\end{eqnarray}
As for the more complicated relations, we first note that using \eqref{qbi1}, \eqref{qbi2}, \eqref{cidef}, \eqref{didef}, we have
\begin{eqnarray}
c_0(l) & = & d_0(l) = (1-q)^{-l},\label{Reduce1}\\
c_l(l) & = & q^{{l+1}\choose{2}}(q-1)^{-l},\label{cup}\\
d_l(l) & = & q^{-{{l}\choose{2}}}(q-1)^{-l}.\label{dup}
\end{eqnarray}
Simple routine computations can be used to show that $p$ divides the integers ${l+1}\choose{2}$ and $-{{l}\choose{2}}$ whenever $l\geq p$, and so \eqref{cup}, \eqref{dup} can be simplified and we have the rule
\begin{eqnarray}
l\geq p \quad \quad& \Rightarrow & \quad\quad c_l(l) = (q-1)^{-l}=d_l(l).\label{Reduce2}
\end{eqnarray}
Also, if $i\in\{1,2,\ldots,p-1\}$, by \eqref{qbireduce1}, \eqref{qbireduce2} we have ${{l}\choose{i}}_q=0$. Then using \eqref{cidef}, \eqref{didef}, we have
\begin{eqnarray}
l\geq p,\   i\in\{1,2,\ldots,p-1\}\quad \quad& \Rightarrow & \quad\quad c_i(l) = 0 =d_i(l).\label{Reduce3}
\end{eqnarray}
By \eqref{Reduce1}, \eqref{Reduce2}, \eqref{Reduce3}, the relations \eqref{equalAlBl}, \eqref{equalBlAl} are now simplified into
\begin{eqnarray}
l\geq p \quad \quad& \Rightarrow & \quad\quad A^lB^l = \frac{\algIcantuba-(-1)^l\LieABcantuba^l}{(1-q)^l}=B^lA^l.\label{reduceAB}
\end{eqnarray}
The relation \eqref{reduceAB} can now be used to derive the parallels of \eqref{ABnbigger}--\eqref{BAnsmaller}. We simply make use of the same pattern of computations described in Section~\ref{structconstSec} except that we use \eqref{medA1p}--\eqref{medB2p} instead of \eqref{medA1}--\eqref{medB2}, and we use \eqref{reduceAB} instead of \eqref{equalAlBl}, \eqref{equalBlAl}. The resulting relations are

\begin{eqnarray}
\LieABcantuba^mA^n\cdot B^l\LieABcantuba^k & = &(1-q)^{-l}\left(q^{\Zcantubamod{(n-l)k}}\algIcantuba-(-1)^{l}q^{\Zcantubamod{(n-l)(m+k)}}\LieABcantuba^l\right)\LieABcantuba^{m+k}A^{n-l},\quad\quad(n\geq l\geq p),\label{ABnbiggerp}\\
\LieABcantuba^mA^n\cdot B^l\LieABcantuba^k & = &(1-q)^{-n}B^{l-n}\LieABcantuba^{m+k}\left(q^{\Zcantubamod{m(l-n)}}\algIcantuba-(-1)^{n}q^{\Zcantubamod{(l-n)(m+n)}}\LieABcantuba^n\right),\quad\  \  (l>n\geq p),\label{ABnsmallerp}\\
B^n\LieABcantuba^m\cdot\LieABcantuba^kA^l & = &q^{\Zcantubamod{l(m+k)}}(1-q)^{-l}B^{n-l}\LieABcantuba^{m+k}\left(\algIcantuba-(-1)^{l}\LieABcantuba^l\right),\quad(n\geq l\geq p),\label{BAnbiggerp}\\
B^n\LieABcantuba^m\cdot\LieABcantuba^kA^l & = &q^{\Zcantubamod{(m+k)n}}(1-q)^{-n}\left(\algIcantuba-(-1)^{n}\LieABcantuba^n\right)\LieABcantuba^{m+k}A^{l-n},\quad(l>n\geq p).\label{BAnsmallerp}
\end{eqnarray}
It is important to note here that the simplified relations \eqref{ABnbiggerp}--\eqref{BAnsmallerp} hold only if the respective inequalities indicated in each of them hold. Otherwise, the original relations \eqref{ABnbigger}--\eqref{BAnsmaller} still prevail.

\subsection{Consequences on the commutators of basis elements}

In this subsection, we discuss the consequences of $q$ being of torsion type on the commutator of two arbitrary basis elements of $\torsionHeisenCantuba$ from \eqref{Heisenbasis}.

\begin{definition}\label{nonLiesubDef} We define $\nonLiesubCantuba$ as the linear subspace of $\torsionHeisenCantuba$ spanned by the basis elements from \eqref{Heisenbasis} of the form
\begin{eqnarray}
B^{m}\LieABcantuba^{n}, \quad \LieABcantuba^{n}A^{m},\nonumber
\end{eqnarray}
where $m,n$ are both nonzero but both are not congruent to zero modulo $p$.
\end{definition}

In all future references, by a basis element of $\nonLiesubCantuba$, we refer to the basis of $\nonLiesubCantuba$ given in Definition~\ref{nonLiesubDef} above. We claim that the linear subspace $\nonLiesubCantuba$ has this interesting property that the commutator of any two basis elements from \eqref{Heisenbasis}, when expressed as a linear combination of \eqref{Heisenbasis}, has no term that is in $\nonLiesubCantuba$. To establish this claim, we need the following.

\begin{lemma}\label{inconstLem} For every $m,k\in\Ncantuba$ and every $n,l\in\Zcantuba^+$, the commutators
\begin{eqnarray}
\lbrackCantuba\LieABcantuba^m,\LieABcantuba^kA^l\rbrackCantuba, & \label{typeLieA}\\
\lbrackCantuba\LieABcantuba^m,B^l\LieABcantuba^k\rbrackCantuba, & \label{typeLieB}\\
\lbrackCantuba\LieABcantuba^mA^n,\LieABcantuba^kA^l\rbrackCantuba, & \label{typeAA}\\
\lbrackCantuba\LieABcantuba^mA^n,B^l\LieABcantuba^k\rbrackCantuba, &\quad\quad(n\neq l) \label{typeAB}\\
\lbrackCantuba B^n\LieABcantuba^m,B^l\LieABcantuba^k\rbrackCantuba, & \label{typeBB}
\end{eqnarray}
when expressed as linear combinations of the basis elements \eqref{Heisenbasis}, do not have a term  which is a basis element of $\nonLiesubCantuba$.
\end{lemma}
\begin{proof} We first tackle the most complicated commutator, which is \eqref{typeAB}. Consider the case $n>  l$. Using \eqref{ABnbigger}, \eqref{BAnsmaller}, we have
\begin{eqnarray}
\lbrackCantuba\LieABcantuba^mA^n,B^l\LieABcantuba^k\rbrackCantuba = \sum_{i=0}^le_i\LieABcantuba^{m+k+i}A^{n-l},
\end{eqnarray}
where $e_i=q^{(k+i)(n-l)}c_i(l)-q^{-(m+k)l}d_i(l)$. Suppose that $i$ assumes some value $j$ such that the corresponding term $$\LieABcantuba^{m+k+j}A^{n-l}$$ is a basis element of $\nonLiesubCantuba$. We prove that $e_j=0$. Since $\LieABcantuba^{m+k+j}A^{n-l}$ is a basis element of $\nonLiesubCantuba$, we have
\begin{eqnarray}
n-l & \equiv & 0\mod p,\label{nlmod}\\
m+k+j & \equiv & 0\mod p.\label{mkjmod}
\end{eqnarray}
Using \eqref{nlmod}, for some integer $N$, we have $n-l=Np$, and so $e_j=(q^p)^{(k+j)N}c_j(l)-q^{-(m+k)l}d_j(l)$, which simplifies into $e_j=c_j(l)-q^{-(m+k)l}d_j(l)$. The condition \eqref{mkjmod} implies that $-(m+k)=j-pJ$ for some integer $J$, which means that we can further simplify $e_j$ as $e_j=c_j(l)-q^{jl}(q^p)^{-J}d_j(l)=c_j(l)-q^{jl}d_j(l)$. At this point, we use \eqref{cidef}, \eqref{didef} to obtain
\begin{eqnarray}
e_j& = &\left(q^{{j+1}\choose{2}}-q^{jl-{l\choose{2}}+{l-j\choose{2}}}\right) (q-1)^{-l}(-1)^{l-j}{{l}\choose{j}}_q.\label{ejsimp}
\end{eqnarray}
We bring the reader's attention to the left-most factor in the right-hand side of \eqref{ejsimp}, which is $q^{{j+1}\choose{2}}-q^{jl-{l\choose{2}}+{l-j\choose{2}}}$. The exponent of $q$ in the second term can be simplified as
\begin{eqnarray}
jl-{l\choose{2}}+{l-j\choose{2}} & = & jl+\frac{1}{2}(l-j)(l-j-1)-\frac{1}{2}l(l-1)\\
& = & \frac{1}{2}j^2+\frac{1}{2}j={j+1\choose{2}}.
\end{eqnarray}
Therefore $q^{{j+1}\choose{2}}-q^{jl-{l\choose{2}}+{l-j\choose{2}}}=q^{{j+1}\choose{2}}-q^{{j+1}\choose{2}}=0$, and so $e_j=0$. The proof for the case $n<l$ is similar, but the relations \eqref{ABnsmaller}, \eqref{BAnbigger} are used instead of \eqref{ABnbigger}, \eqref{BAnsmaller}. Next, we consider commutators of type \eqref{typeAA} and \eqref{typeBB}. Using the methods in Section~\ref{structconstSec}, we have
\begin{eqnarray}
\lbrackCantuba\LieABcantuba^mA^n,\LieABcantuba^kA^l\rbrackCantuba & = &  q^{nk}(1-q^{ml-nk})\LieABcantuba^{m+k}A^{n+l},\label{AAsimp1}\\
\lbrackCantuba B^n\LieABcantuba^m,B^l\LieABcantuba^k\rbrackCantuba & = & q^{nk}(q^{ml-nk}-1) B^{n+l}\LieABcantuba^{m+k}.\label{BBsimp1}
\end{eqnarray}
If $\LieABcantuba^{m+k}A^{n+l}$ or $B^{n+l}\LieABcantuba^{m+k}$ is a basis element of $\nonLiesubCantuba$, then we have the condition that both $m+k$ and $n+l$ are congruent to zero modulo $p$. Then for some integers $M,N$, we have $k=Mp-m$ and $n=Np-l$. By routine computations, we have $ml-nk=p(Nm+Ml-MNp)$. Thus, \eqref{AAsimp1}, \eqref{BBsimp1} become
\begin{eqnarray}
\lbrackCantuba\LieABcantuba^mA^n,\LieABcantuba^kA^l\rbrackCantuba & = &  q^{nk}(1-(q^p)^{Nm+Ml-MNp})\LieABcantuba^{m+k}A^{n+l},\nonumber\\
\lbrackCantuba B^n\LieABcantuba^m,B^l\LieABcantuba^k\rbrackCantuba & = & q^{nk}((q^p)^{Nm+Ml-MNp}-1) B^{n+l}\LieABcantuba^{m+k},\nonumber
\end{eqnarray}
from which we see that, since $q^p=1$, the coefficients of any basis element of $\nonLiesubCantuba$ is zero in the linear combinations for the commutators \eqref{typeAA}, \eqref{typeBB}. Finally, we have the commutators of type \eqref{typeLieA}, \eqref{typeLieB}. Using the techniques in Section~\ref{structconstSec}, we have
\begin{eqnarray}
\lbrackCantuba\LieABcantuba^m,\LieABcantuba^kA^l\rbrackCantuba & = & (1-q^{lm})\LieABcantuba^{m+k}A^l,\label{LieAsimp}\\
\lbrackCantuba\LieABcantuba^m,B^l\LieABcantuba^k\rbrackCantuba & = & -(1-q^{lm})B^l\LieABcantuba^{m+k}\label{LieBsimp}.
\end{eqnarray}
If we assume that $\LieABcantuba^{m+k}A^l$ or $B^l\LieABcantuba^{m+k}$ is a basis element of $\nonLiesubCantuba$, then $l$ and $m+k$ are both congruent to zero modulo $p$. Then $q^{lm}=1$, and so the coefficients of $\LieABcantuba^{m+k}A^l$ and $B^l\LieABcantuba^{m+k}$ in \eqref{LieAsimp}, \eqref{LieBsimp} are both zero. This completes the proof.\qed
\end{proof}

\begin{corollary}\label{TheCor} The sum of the linear subspaces $\torsionLieCantuba$ and $\nonLiesubCantuba$ of $\torsionHeisenCantuba$ is direct.
\end{corollary}
\begin{proof} This is equivalent to saying that every Lie polynomial in $A,B$, when written as a linear combination of \eqref{Heisenbasis}, does not have a term which is a basis element of $\nonLiesubCantuba$. This clearly holds for the generators $A,B$. We are done if we show that given Lie polynomials $u,v$, the commutator $\lbrackCantuba u,v\rbrackCantuba$, when written as a linear combination of \eqref{Heisenbasis}, does not have a term which is a basis element of $\nonLiesubCantuba$. Each of $u,v$, anyway, is a linear combination of \eqref{Heisenbasis}, and so we are done if we show that the commutator of any two basis elements from \eqref{Heisenbasis}, when written as a linear combination of \eqref{Heisenbasis}, does not have a term which is a basis element of $\nonLiesubCantuba$. For this, we need the commutator table for the elements \eqref{Heisenbasis} of $\torsionHeisenCantuba$, and in each cell of the commutator table, instead of showing the lengthy formula, we give the justification of why the resulting commutator does not have a term that is in $\nonLiesubCantuba$. The commutators in Lemma~\ref{inconstLem} cover all such cases, as shown in the table below:\\
\begin{center}
\scalebox{0.73}{
\begin{tabular}{|c|c|c|c|}
\hline
$\lbrackCantuba\cdot,\cdot\rbrackCantuba$ &  $\LieABcantuba^k$ & $\LieABcantuba^kA^l$ & $B^l\LieABcantuba^k$  \\
\hline
$\LieABcantuba^m$  & 0 & \eqref{typeLieA} & \eqref{typeLieB} \\
\hline
$\LieABcantuba^mA^n$  & - & \eqref{typeAA} & \eqref{typeAB} \\
\hline
$B^n\LieABcantuba^m$  & - & - & \eqref{typeBB} \\
\hline
\end{tabular}}
\captionof{table}{Taking the commutator of any two basis elements from \eqref{Heisenbasis} results to a linear combination which does not have a term that is a basis element of $\nonLiesubCantuba$\\}\label{table:comm}
%\end{table}
\end{center}
Some cells in Table~\ref{table:comm} are marked with a dash because of the skew-symmetry of the commutator as a Lie bracket operation, while the cell with $0$ is because any two powers of $\LieABcantuba$ commute. The equation numbers in the other cells indicate the commutator considered in the proof of Lemma~\ref{inconstLem} which does not have a term that is a basis element of $\nonLiesubCantuba$. We note here that \eqref{typeAB} has the restriction $n\neq l$, but we further justify that even if we have $n=l$, by the methods in Section~\ref{structconstSec}, the result is a polynomial in $\LieABcantuba$ and in such an expression, none of the basis elements of $\nonLiesubCantuba$ has a nonzero coefficient. This completes the proof.\qed
\end{proof}

\section{The Lie algebra $\torsionLieCantuba$}

In this section, we construct a Lie subalgebra of $\torsionHeisenCantuba$ with a given basis, and we show that this is isomorphic to $\torsionLieCantuba$.

\begin{definition}\label{preLieDef} We define $\preHLieCantuba$ as the $\Fcantuba$-linear subspace of $\torsionHeisenCantuba$ spanned by the following elements.
\begin{eqnarray}\label{pHeisenLiebasis}
A,\  B,\quad\LieABcantuba^n, &\\
\quad\LieABcantuba^k A^l,\quad B^{l}\LieABcantuba^k, &\label{resexp}
\end{eqnarray}
where $n,k,l\in\Zcantuba^+$ such that $n\geq 2$ implies $n-1\  {\not\equiv}\  0\mod p$, and that in each element of the form \eqref{resexp}, at least one of $k,l$ is not congruent to zero modulo $p$.  We define $\derLieCantuba$ as the linear subspace of $\preHLieCantuba$ spanned by all elements in \eqref{pHeisenLiebasis} except $A,B$.
\end{definition}

Since the elements in \eqref{pHeisenLiebasis} are among \eqref{Heisenbasis} that are linearly independent in $\torsionHeisenCantuba$, we find that the elements in \eqref{pHeisenLiebasis} form a basis for $\preHLieCantuba$, while the elements in \eqref{pHeisenLiebasis} except $A,B$ form a basis for $\derLieCantuba$.

\def\colIII{\LieABcantuba^r}
\def\colIV{\LieABcantuba^rA^s}
\def\colV{B^s\LieABcantuba^r}

\def\rowIII{\LieABcantuba^m}
\def\rowIV{\LieABcantuba^mA^n}
\def\rowV{B^n\LieABcantuba^m}

\begin{lemma}\label{expsLem} If $m,n,r,s\in\Zcantuba^+$, then $\lbrackCantuba \rowIV, \colV\rbrackCantuba$ is a linear combination of elements in \eqref{pHeisenLiebasis}, in all of which, the exponents of $\LieABcantuba$ are always positive.
\end{lemma}
\begin{proof} The commutator $f:=\lbrackCantuba \rowIV, \colV\rbrackCantuba=\LieABcantuba^mA^n\cdot B^n\LieABcantuba^r-B^n\LieABcantuba^{m}\cdot \LieABcantuba^{r}A^n $ can be simplified using the techniques discussed in Section~\ref{structconstSec}. Based on such routine computations, we have

\begin{eqnarray}
f=\sum_{i=0}^n(c_i(n)-q^{n(m+r)}d_i(n))\LieABcantuba^{m+i+r},\quad\quad (n=s),\nonumber\\
 f=\sum_{i=0}^n(q^{(s-n)(m+n)}c_i(n)-q^{n(m+r)}d_i(n))B^{s-n}\LieABcantuba^{m+i+r},\quad\quad (n<s),\nonumber\\
f=\sum_{i=0}^s(q^{(n-s)(i+r)}c_i(s)-q^{s(m+r)}d_i(s))\LieABcantuba^{m+i+r}A^{n-s},\quad\quad (n>s).\nonumber
\end{eqnarray}
In any of the above cases, the elements from \eqref{Heisenbasis} that appear in the linear combination for $f$ involve only those with $\LieABcantuba^{m+i+r}$ where $m+i+r\geq 2$ for all $i$.\qed
\end{proof}

\begin{lemma}\label{LiesubLem} If $f,g$ are any two of the basis elements of $\preHLieCantuba$ in \eqref{pHeisenLiebasis}, then $\lbrackCantuba f,g\rbrackCantuba\in\derLieCantuba$.
\end{lemma}
\begin{proof} The simplest case is $f=A$ and $g=B$, where clearly the commutator of these two elements is in $\derLieCantuba$. Let $k,l,r,s\in\Zcantuba^+$. If $\lbrackCantuba f,g\rbrackCantuba$ is one of
\begin{eqnarray}
\lbrackCantuba A,\colIII\rbrackCantuba,\  \lbrackCantuba A,\colIV\rbrackCantuba,\  \lbrackCantuba \rowIII,\colIV\rbrackCantuba,\  \lbrackCantuba \rowIV,\colIV\rbrackCantuba,
\end{eqnarray}
then by the techniques in Section~\ref{structconstSec}, mainly with the use of the relation \eqref{moveAi}, it can be shown that $\lbrackCantuba f,g\rbrackCantuba$ is in the span of elements of the form $\LieABcantuba^kA^l$ for all $k,l\in\Zcantuba^+$, where at least one of $k,l$ is not congruent to zero modulo $p$, because of Corollary~\ref{TheCor}. Similarly, if $\lbrackCantuba f,g\rbrackCantuba$ is one of the following
\begin{eqnarray}
\lbrackCantuba B,\colIII\rbrackCantuba,\  \lbrackCantuba B,\colV\rbrackCantuba,\  \lbrackCantuba \rowIII,\colV\rbrackCantuba,\  \lbrackCantuba \rowV,\colV\rbrackCantuba,
\end{eqnarray}
then by using the relation \eqref{moveBj} as described in Section~\ref{structconstSec}, $\lbrackCantuba f,g\rbrackCantuba$ is in the span of elements of the form $B^l\LieABcantuba^k$ for all $k,l\in\Zcantuba^+$, where, again, at least one of $k,l$ is not congruent to zero modulo $p$, because of Corollary~\ref{TheCor}.. The case $\lbrackCantuba f,g\rbrackCantuba=\lbrackCantuba\rowIV ,\colV\rbrackCantuba$ has been dealt with in Lemma~\ref{expsLem}. Consider the case $f=B$ and $g=\colIV$. Using the techniques in Section~\ref{structconstSec}, but mainly through the relations \eqref{moveBj}, \eqref{equalAlBl}, \eqref{equalBlAl}, it is routine to show that $\lbrackCantuba f,g\rbrackCantuba$ is a linear combination of $\LieABcantuba^{r+1}A^{s-1}$ and $\LieABcantuba^rA^{s-1}$, and so $\lbrackCantuba f,g\rbrackCantuba\in\derLieCantuba$. We have a similar proof that $\lbrackCantuba f,g\rbrackCantuba\in\derLieCantuba$ if $f=A$ and $g=\colV$. In these last two cases, we note that Corollary~\ref{TheCor} is in effect. All other cases for $f,g$ not mentioned in the above would either result to $0$ or to $-\lbrackCantuba f',g'\rbrackCantuba$ where $f',g'$ satisfy one of the previously mentioned cases. This completes the proof.\qed
\end{proof}

\begin{lemma}\label{subsetLem} Each basis element of $\preHLieCantuba$ in \eqref{pHeisenLiebasis} is an element of $\torsionLieCantuba$.
\end{lemma}
\begin{proof} The elements $A,B$, and $\LieABcantuba$ are clearly in $\torsionLieCantuba$. To prove that the other elements in \eqref{pHeisenLiebasis} are in $\torsionLieCantuba$, or equivalently to construct them as Lie polynomials, we borrow some of the constructions in \cite[Section 5]{Can17}. Let $k\in\Ncantuba$ and $l\in\Zcantuba^+$. We define the following elements of $\torsionLieCantuba$.

\begin{eqnarray}
\baseAcantuba(k,l) & := &  \left(\left(-\adCantuba \LieABcantuba\right)^k\circ\left(-\adCantuba A\right)^{l+1}\right)\left( B\right),\\
\baseBcantuba(k,l) & := &  \left(\left(\adCantuba B\right)^{l-1}\circ\left(\adCantuba \LieABcantuba\right)^{k}\right)\left( \lbrackCantuba B,\lbrackCantuba B,A\rbrackCantuba\rbrackCantuba\right),\\
\baseGcantuba(k) & := & \left(\left(\adCantuba B\right)\circ\left(-\adCantuba \LieABcantuba\right)^{k}\right)\left( \lbrackCantuba\lbrackCantuba B,A\rbrackCantuba,A\rbrackCantuba\right).
\end{eqnarray}

The following relations can be derived using the techniques in Section~\ref{structconstSec}.

\begin{eqnarray}
\baseAcantuba(k,l) & = & -(1-q)^l(q^l-1)^k\LieABcantuba^{k+1}A^l,\label{baseArel}\\
\baseBcantuba(k,l) & = & (q-1)^{k+1}(1-q^{k+1})^{l-1}B^l\LieABcantuba^{k+1},\label{baseBrel}\\
\baseGcantuba(k) & = & q^{-k}(q-1)^{k+1}\left( q\{k\}_q\LieABcantuba^{k+1}-\{k+1\}_q\LieABcantuba^{k+2}\right),\label{baseGrel}\\
q^k\sum_{i=0}^k(q-1)^{-(i+1)}\baseGcantuba(i) & = & -\{k+1\}_q\LieABcantuba^{k+2}.\label{baseGrel2}
\end{eqnarray}

The next goal is to isolate elements of the form $\LieABcantuba^{k+1}A^l$, $B^l\LieABcantuba^{k+1}$, and $\LieABcantuba^{k+2}$ in the right-hand sides of \eqref{baseArel}, \eqref{baseBrel}, \eqref{baseGrel2}, and so we define the following.

\begin{eqnarray}
\obaseAcantuba(k,l) \quad := & -(1-q)^{-l}(q^l-1)^{-k}\baseAcantuba(k,l) & = \quad \LieABcantuba^{k+1}A^l,\quad (l{\  \not\equiv\  } 0\mod p),\label{obaseAdef}\\
\obaseBcantuba(k,l) \quad := & (q-1)^{-k-1}(1-q^{k+1})^{1-l}\baseBcantuba(k,l) & = \quad B^l\LieABcantuba^{k+1},\quad  (k+1{\  \not\equiv\  } 0\mod p),\label{obaseBdef}\\
\obaseGcantuba(k) \quad := & \frac{q^k}{1-q^{k+1}}\sum_{i=0}^k(q-1)^{-i}\baseGcantuba(i) & = \quad \LieABcantuba^{k+2},\quad (k+1{\  \not\equiv\  } 0\mod p).\label{obaseGdef}
\end{eqnarray}

With reference to \eqref{obaseAdef}, \eqref{obaseBdef}, we show how to construct $\LieABcantuba^{k+1}A^l$ and $\quad B^l\LieABcantuba^{k+1}$ as Lie polynomials for the cases $l\equiv 0\mod p$ and $k+1\equiv 0\mod p$, respectively. The trick is that in such cases, we have $l-1{\  \not\equiv\  }  0\mod p$ and $k{\  \not\equiv\  } 0\mod p$, respectively, and so by \eqref{obaseAdef}, \eqref{obaseBdef}, the Lie polynomial constructions $\obaseAcantuba(k,l-1)$ and $\obaseBcantuba(k-1,l)$ are possible in each respective case. We use the techniques in Section~\ref{structconstSec} to obtain
\begin{eqnarray}
\lbrackCantuba \obaseAcantuba(k,l-1),A\rbrackCantuba & = & (1-q^{k+1}) \LieABcantuba^{k+1}A^l,\quad (l{\  \equiv\  } 0\mod p,\quad  k+1{\  \not\equiv\  } 0\mod p),\label{special1}\\
\lbrackCantuba \obaseBcantuba(k-1,l),\LieABcantuba\rbrackCantuba & = & (1-q^l) B^l\LieABcantuba^{k+1},\quad  (k+1{\  \equiv\  } 0\mod p,\quad l{\  \not\equiv\  } 0\mod p).\label{special2}
\end{eqnarray}
In \eqref{special1}, the condition $k+1{\  \not\equiv\  } 0\mod p$ is needed. Otherwise, we have a basis element of $\nonLiesubCantuba$ instead of an element in \eqref{pHeisenLiebasis}. This, in turn, ensures that the scalar coefficient $1-q^{k+1}$ is nonzero and our Lie polynomial construction is successful. Similar reasoning applies to \eqref{special2}. What remains to be shown is that, with reference to \eqref{obaseGdef}, how to construct a nested commutator equal to $\LieABcantuba^{k+2}$ when $k+1{\   \equiv\  } 0\mod p$ (where we note that $k\geq 1$ because $p\geq 2$). For this case, by \cite[Corollary 6.12]{Hel}, both $A^{k+1}$ and $B^{k+1}$ are central in $\torsionHeisenCantuba$. Then by \eqref{reduceAB}, so is $\kappa:=\LieABcantuba^{k+1}$. Since $1$ is not congruent to zero modulo $p$, by \eqref{obaseAdef}, \eqref{obaseBdef}, both $\LieABcantuba^{k}A$ and $B\LieABcantuba$ are Lie polynomials in $A,B$. Then so is $\lbrackCantuba \LieABcantuba^{k}A, B\LieABcantuba\rbrackCantuba=\LieABcantuba^kAB\LieABcantuba-B\kappa A=q^{-1}\cdot q\LieABcantuba^{k+1}AB-B\kappa A=\kappa(AB-BA)=\LieABcantuba^{k+2}$. This completes the Lie polynomial constructions needed in the proof.\qed
\end{proof}

\begin{theorem}\begin{enumerate}
\item\label{smallThm} As Lie algebras, $\preHLieCantuba$ and $\torsionLieCantuba$ are isomorphic.
\item\label{smallCor} The derived (Lie) algebra of $\torsionLieCantuba$ is $\derLieCantuba$.
\end{enumerate}
\end{theorem}
\begin{proof} By Lemmas~\ref{LiesubLem} and \ref{subsetLem}, $\preHLieCantuba$ is a Lie subalgebra of $\torsionHeisenCantuba$ that contains $A,B$ and is contained in $\torsionLieCantuba$. Therefore, $\preHLieCantuba=\torsionLieCantuba$. To prove \ref{smallCor}, use Lemma~\ref{LiesubLem} and Theorem~\ref{smallThm}.\qed
\end{proof}

\subsection{Lie polynomials in each $\Zcantuba$-gradation subspace: a summary}

We now characterize the Lie polynomials in each of the $\Zcantuba$-gradiation subspaces in $\{\gradsubCantuba_n\  :\  n\in\Zcantuba\}$.  For each gradation subspace $\gradsubCantuba_n$, we simply identify which basis elements of $\gradsubCantuba_n$ described in the end of Section~\ref{structconstSec} are among the defining basis elements of the Lie algebra  $\preHLieCantuba=\torsionLieCantuba$ from Definition~\ref{preLieDef}. We write in concise form all the nested commutator constructions in the proof of Lemma~\ref{subsetLem}. Given $k\in\Ncantuba$ and $n,l\in\Zcantuba^+$, these nested commutators are

\begin{eqnarray}
\LieABcantuba^{k+2} & = & \frac{-q^k(1-q)}{1-q^{k+1}}\sum_{i=0}^k\frac{\left(\left(\adCantuba B\right)\circ\left(-\adCantuba \LieABcantuba\right)^{k}\right)\left( \lbrackCantuba\lbrackCantuba B,A\rbrackCantuba,A\rbrackCantuba\right)}{(q-1)^{1+i}},\quad  (k+1{\  \not\equiv\  } 0\mod p),\label{combase1}\\
\LieABcantuba^{k+2} & = & \lbrackCantuba\frac{\left(\left(-\adCantuba \LieABcantuba\right)^{k-1}\circ\left(-\adCantuba A\right)^{2}\right)\left( B\right)}{(q-1)^k},\frac{ \lbrackCantuba B,\lbrackCantuba B,A\rbrackCantuba\rbrackCantuba}{q-1}\rbrackCantuba,\quad\quad\    ( k+1{\  \equiv\  } 0\mod p),\label{combase1a}\\
\LieABcantuba^{k+1}A^l & = & -\frac{\left(\left(-\adCantuba \LieABcantuba\right)^k\circ\left(-\adCantuba A\right)^{l+1}\right)\left( B\right)}{(1-q)^l(q^l-1)^k}, \quad ( l{\  \not\equiv\  } 0\mod p),\label{combase2}\\
\LieABcantuba^{k+1}A^{np} & = & \frac{(1-q)^{1-np}}{(1-q^{k+1})(q^{np-1}-1)^k}\left((\adCantuba A)\circ\left(-\adCantuba \LieABcantuba\right)^k\circ\left(-\adCantuba A\right)^{np}\right)\left( B\right),\    (k+1{\  \not\equiv\  } 0\mod p),\label{combase3}\\
B^l\LieABcantuba^{k+1} & = & \frac{\left(\left(\adCantuba B\right)^{l-1}\circ\left(\adCantuba \LieABcantuba\right)^{k}\right)\left( \lbrackCantuba B,\lbrackCantuba B,A\rbrackCantuba\rbrackCantuba\right)}{(q-1)^{k+1}(1-q^{k+1})^{l-1}},\quad  (k+1{\  \not\equiv\  } 0\mod p),\label{combase4}\\
B^{l}\LieABcantuba^{np} & = & \frac{ \left((-\adCantuba\LieABcantuba)\circ\left(\adCantuba B\right)^{l-1}\circ\left(\adCantuba \LieABcantuba\right)^{np-2}\right)\left( \lbrackCantuba B,\lbrackCantuba B,A\rbrackCantuba\rbrackCantuba\right)}{(1-q^l)(q-1)^{np-1}(1-q^{np-1})^{l-1}}, \quad ( l{\  \not\equiv\  } 0\mod p).\label{combase5}
\end{eqnarray}
We note that in \eqref{combase1a}, the composition of maps $(-\adCantuba \LieABcantuba)^{k-1}$ is well defined. We cannot have $k=0$ in this case. Otherwise, we have $1{\  \equiv\  } 0\mod p$, which is not possible because $p\geq 2$. The Lie polynomial characterizations per gradation subspace are summarized in Table~\ref{table:LPC}.\\

\begin{center}
\scalebox{0.73}{
\begin{tabular}{|c|c|c|c|}
\hline
$\Zcantuba$-gradation subspace &  Basis elements & Basis elements that are NOT Lie polynomials in $A,B$ & Basis elements that are Lie polynomials in $A,B$ \\
\hline
$\gradsubCantuba_0$  & $\LieABcantuba^k$ ($k\in\Ncantuba$) & $\algIcantuba$ & $\LieABcantuba$, \eqref{combase1}, \eqref{combase1a} \\
\hline
$\gradsubCantuba_{-1}$  & $\LieABcantuba^kA$ ($k\in\Ncantuba$) & (none) & $A$, \eqref{combase2} ($l=1$) \\
\hline
$\gradsubCantuba_{-l}$\quad  ($l\geq 2$, $l{\  \not\equiv\  } 0 \mod p$)  & $\LieABcantuba^kA^l$ ($k\in\Ncantuba$) & $A^l$  & \eqref{combase2}  \\
\hline
$\gradsubCantuba_{-np}$\quad ($n\in\Zcantuba^+$)  & $\LieABcantuba^kA^{np}$ ($k\in\Ncantuba$)  &  $\LieABcantuba^hA^{np}$,  $(h\equiv0\mod p)$ & \eqref{combase3}  \\
\hline
$\gradsubCantuba_{1}$  & $B\LieABcantuba^k$ ($k\in\Ncantuba$) & (none) & $B$, \eqref{combase4} ($l=1$) \\
\hline
$\gradsubCantuba_{l}$\quad ($l\geq 2$, $l{\  \not\equiv\  } 0 \mod p$ )  & $B^l\LieABcantuba^k$ ($k\in\Ncantuba$) &  $B^l$ & \eqref{combase4}  \\
\hline
$\gradsubCantuba_{np}$\quad ($n\in\Zcantuba^+$)  & $B^{np}\LieABcantuba^k$ ($k\in\Ncantuba$)  & $B^{np}\LieABcantuba^h$,  $(h\equiv0\mod p)$ & \eqref{combase5}  \\
\hline
\end{tabular}}
\captionof{table}{Lie polynomial characterization for each $\Zcantuba$-gradation subspace of $\torsionHeisenCantuba$\\}\label{table:LPC}
%\end{table}
\end{center}
Let $m\in\Zcantuba$. The information from Table~\ref{table:LPC} about each basis element is enough to characterize whether an arbitrary $u\in\gradsubCantuba_m$ is a Lie polynomial in $A,B$ or not. Write $u$ as a linear combination $u=\sum_{i=1}^tc_i\beta_i$ where, for any $i$, $c_i\neq 0$ and $\beta_i$ is a basis element of $H_m$ from the second column of Table~\ref{table:LPC}. If exactly one $\beta_j$ is not a Lie polynomial in $A,B$, then so is $u$. Otherwise, $c_j^{-1}\left(u-\sum_{i\neq j}^tc_i\beta_i\right)=\beta_j$ is a Lie polynomial in $A,B$, which is a contradiction. This argument can be extended to any case when there is more than one element of $\{\beta_i\  :\  i\in\{1,2,\ldots,t\}\}$ that is not a Lie polynomial in $A,B$.

The study of Lie polynomials, such as that done in \cite{Can17} and in this work, which are in associative algebras defined by generators and relations where the relations are not all Lie polynomials in the generators was begun in the study \cite{Can} for another algebra in which the Lie polynomials have not been fully characterized. As the study \cite{Can17} and this work show, the $q$-deformed Heisenberg commutation relation $AB-qBA=\algIcantuba$ implies a rich structure and results about the Lie polynomials in $A,B$. The authors hope that to the interested reader, who has been studying other associative algebras with more varied presentations, may find this kind of study on such other algebras worth considering.

\begin{acknowledgement}
The authors are grateful to the International Mathematical Union for support of their research and Rafael Cantuba's visit to M\"alardalen University in Autumn 2018, and to the Mathematics and Applied Mathematics research environment, the School of Education, Culture and Communication at M\"alardalen University for support and excellent environment for research in Mathematics, and also for the support from De~La Salle University, Manila.
\end{acknowledgement}

\end{document}